\documentclass[oneside,10pt]{article}        
\usepackage[b5paper]{geometry}	 
\usepackage{amsfonts,amsmath,latexsym,amssymb} 
\usepackage{theorem}                
\usepackage{enumerate}
\usepackage{mathrsfs,upref}       
\usepackage{mathptmx}		    
\usepackage{dea}	            
\newtheorem{theorem}{Theorem}           
\newtheorem{lemma}{Lemma}               
\newtheorem{corollary}{Corollary}

\newcommand{\mc}{\mathcal{C}}
\newcommand{\mr}{\mathbb{R}}

\theoremstyle{definition}

\newtheorem{remark}{Remark}

\begin{document}

\title[Short title]{On the Solvability of Nonlinear Differential Equations Subject to Generalized Boundary Conditions}

\author{Benjamin Freedman and Jes\'{u}s Rodr\'{i}guez}

\address{Benjamin Freedman \\ 
Department of Mathematics \\
Box 8205, NCSU, Raleigh, NC 27695-8205\\
USA \\
\email{bnfreedm@ncsu.edu}}

\address{Jes\'{u}s Rodr\'{i}guez \\
Department of Mathematics \\
Box 8205, NCSU, Raleigh, NC 27695-8205 \\
USA \\
\email{rodrigu@ncsu.edu}}

\CorrespondingAuthor{Jes\'{u}s Rodr\'{i}guez}

\date{28.02.2018}                               

\keywords{boundary value problems, ordinary differential equation, nonlinear equations, fixed-point theorems}

\subjclass{34A34, 34B15, 47H09, 47H10, 47J07}

\begin{abstract}
        In this paper, we analyze nonlinear differential equations subject to generalized boundary conditions. More specifically, we provide a framework from which we can provide conditions, which are straightforward to check, for the solvability of a large number of nonlinear scalar boundary value problems. We begin by giving our general strategy which involves the reformulation of our boundary value problem as an operator equation. We then proceed to establish our results and compare them to closely related previous work. 
\end{abstract}

\maketitle

\section{Introduction}
This paper is devoted to the study of nonlinear differential equations subject to generalized nonlinear boundary conditions or constraints. The class of problems we consider include, as a special case, differential equations subject to multi-point boundary conditions. A framework is provided which enables us to establish easily verifiable conditions which guarantee the existence of solutions to a significant class of problems. Two relevant papers devoted to the study of nonlinear differential equations subject to constraints are \cite{ja2} and \cite{jsuar1}. The work that we will now present allows us to establish the existence of solutions for problems that do lie within the scope of the results in either one of these two papers. 

The literature concerning the study of boundary value problems is extensive. In \cite{stiel}, \cite{rodpad} and \cite{urabe} the authors analyze boundary value problems subject to linear constraints. In \cite{AN}, the reader will find results pertaining to three-point boundary value problems. The use of projection schemes appears in \cite{mar}, \cite{rod2}, \cite{stiel}, \cite{rodpad} and \cite{urabe}.

In \cite{KB}, \cite{BL}, \cite{ja}, and \cite{ja2} the authors obtain existence results based on a global inverse function theorem. Readers interested in fractional differential equations may consult \cite{AHM} and those who would like to see results involving discrete-time systems are referred to \cite{ma}, \cite{rod3}, \cite{ja} and \cite{jsuar1}.

\section{Generalities}
\indent In this paper, we study nonlinear scalar boundary value problems which we approach by reformulating as an operator equation of the form
\begin{align}
\mathcal{L}x=H(x) \label{eqmain}
\end{align}
where $\mathcal{L}$ is a linear operator, $H$ is a nonlinear operator and both are defined on a Banach space. Suppose that $\mathcal{L}$ has an inverse, and $H=\Psi+G$. The strategy we will employ is to first give conditions under which $\mathcal{L}-\Psi$ is guaranteed an inverse. That is, we give conditions under which we can uniquely solve the equation
\begin{align}
\mathcal{L}x-\Psi(x)=y \label{eqsplit}
\end{align}
for any point $y$ in the space that $\mathcal{L}$ and $\Psi$ map into. Given a result of this type, we then study conditions under which $\eqref{eqmain}$ has a (possibly non-unique) solution by studying the operator $(\mathcal{L}-\Psi)^{-1} G$ and determining conditions under which it has a fixed point. This will rely on a Schauder's fixed point theorem argument. \\

\section{Differential Equations}

        We consider nonlinear differential equations on the interval $[0,1]$ of the form
\begin{align}
a_n(t)x^{(n)}(t)+a_{n-1}(t)x^{(n-1)}(t)+\dots+a_0(t) x(t)+\psi(x(t))=G(x)(t) \label{de}
\end{align}
subject to the boundary conditions
\begin{align}
\sum_{j=1}^n\int_0^1 x^{(j-1)}(t)d\omega_{ij}(t)+\eta_i(x) &=\phi_i(x)  \label{bc}
\end{align}
\\
for $1 \leq i \leq n$. \\ \\
\indent Here $\psi: \mathbb{R} \to \mathbb{R}$, the maps $\eta_i$ and $\phi_i$ for $1 \leq i \leq n$ are nonlinear real-valued maps from $\mc=(C[0,1],\mathbb{R},\| \cdot \|_{\infty})$ into $\mathbb{R}$ where $\| \cdot \|_{\infty}$ denotes the supremum norm. Further $G: \mc \to \mc$ is a continuous map, $a_0, a_1, \dots, a_n \in \mc$ and $a_n(t) \neq 0$ for all $t \in [0,1]$. We use $\mc^n$ to denote the subspace of $\mc$ consisting of all $n$-times continuously differentiable functions on $[0,1]$. For $1 \leq i \leq n$, $1 \leq j \leq n$, $\omega_{ij}: [0,1] \to \mathbb{R}$ is a real-valued function of bounded variation. We will determine conditions under which we can guarantee at least one solution in $\mc^n$ to $\eqref{de}-\eqref{bc}$. \\ \\
\indent To do so, we first consider a closely related problem. That is, we seek conditions under which we can uniquely solve 
\newpage
\begin{align}
a_n(t)x^{(n)}(t)+a_{n-1}(t)x^{(n-1)}(t)+\dots+a_0(t) x(t)+\psi(x(t))=h(t) \nonumber
\end{align}
subject to
\begin{align}
\sum_{j=1}^n\int_0^1 x^{(j-1)}(t)d\omega_{ij}(t) +\eta_i(x)&=v_i \nonumber
\end{align}
for $1 \leq i \leq n$ and for any $h \in \mc$ and $v \in \mathbb{R}^n$. \\

Define $L: \mc^n \to \mc$ by 
\begin{align}
[Lx](t)=a_n(t)x^{(n)}(t)+a_{n-1}(t)x^{(n-1)}(t)+\dots+a_0(t) x(t) \nonumber
\end{align}
and $B: \mc^n \to \mathbb{R}^n$ by

\begin{align}
B(x)=\begin{bmatrix}
\sum_{j=1}^n\int_0^1 x^{(j-1)}(t)d\omega_{1j}(t) \\
\sum_{j=1}^n\int_0^1 x^{(j-1)}(t)d\omega_{2j}(t) \\
\cdots \\
\sum_{j=1}^n\int_0^1 x^{(j-1)}(t)d\omega_{nj}(t)
\end{bmatrix}.
\nonumber
\end{align}

The map $\mathcal{L}:\mc^n \to \mc \times \mathbb{R}^n$ will be defined as
\begin{align}
\mathcal{L}=\Bigg(
\begin{array}{c}
L \\
B 
\end{array}
\Bigg). \nonumber
\end{align} \\

Before proceeding, we state the following remark which illustrates an important special case of $\eqref{de}-\eqref{bc}$ that can be dealt with using the framework of this section.

\begin{remark}
Let $t_0 \in [0,1]$, $\beta \in \mathbb{R}$ and let the function $\omega: [0,1] \to \mathbb{R}$ be the step function
\[
\omega(x) =\begin{cases}
\displaystyle 0 , & t < t_0 ,\\[3 mm ]
\beta, & t \geq t_0,
\end{cases}
\]

So for any $x \in \mc$, the Riemann-Stieltjes of $x$ with respect to $\omega$ is given by
\begin{align}
\int_0^1 x(t)d\omega(t) =\beta x(t_0). \nonumber
\end{align}
\newpage
Therefore, the boundary value problem $\eqref{de}-\eqref{bc}$ includes problems of the form
\begin{align}
a_n(t)x^{(n)}(t)+a_{n-1}(t)x^{(n-1)}(t)+\dots+a_0(t) x(t)+\psi(x(t))=G(x)(t) \nonumber
\end{align}
subject to the multipoint boundary conditions

\begin{align}
\sum_{i=1}^q B_i \bar{x}(t_i)+\begin{bmatrix}
           \eta_1(x) \\
           \eta_2(x) \\
           \dots \\
           \eta_n(x)
         \end{bmatrix}=
         \begin{bmatrix}
           \phi_1(x) \\
           \phi_2(x) \\
           \dots \\
           \phi_n(x)
         \end{bmatrix} \nonumber
\end{align} 

where
\begin{align}
\bar{x}(t)=\begin{bmatrix}
x(t) \\
x'(t) \\
\cdots \\
x^{(n-1)}(t) 
\end{bmatrix} \nonumber
\end{align}

$t_i \in [0,1]$, and $B_i$ is a real-valued $n \times n$ matrix for all $1 \leq i \leq q$.
\end{remark}

It is well known from the theory of linear differential equations that $\ker(L)$ is $n$-dimensional. Without loss of generality, choose a basis $\{u_1,u_2, \dots, u_n\}$ for the kernel of $L$ such that $\|u_1\|+\|u_2\|+ \dots +\|u_n\| \leq 1$ and let

\begin{align}
u=\begin{bmatrix}
u_1 \\
u_2 \\
\cdots \\
u_n
\end{bmatrix} \nonumber
\end{align}

Suppose that the functions of bounded variation $\omega_{ij}:[0,1] \to \mathbb{R}$ for $i, j \in \{1,2, \dots, n\}$ appearing in the boundary conditions are such that the $n \times n$ real-valued matrix $\mathcal{B}=[Bu_1|Bu_2| \dots| Bu_n]$ is invertible. \\

\indent Defining the constant $B_0$ as: 
\begin{align}
B_0&=\|\mathcal{B}^{-1}\| \nonumber
\end{align}
then we have that for any $v \in \mr^n$
\begin{align}
\|u^T \mathcal{B}^{-1} v\| &\leq B_0 |v|. \nonumber
\end{align}
\\
\indent As a matter of notation, define $\eta: \mc \to \mathbb{R}^n$ by $\eta(x)=\begin{bmatrix}
\eta_1(x) \nonumber \\
\eta_2(x) \nonumber \\
\dots \nonumber \\
\eta_n(x) \nonumber 
\end{bmatrix}$. \\

\begin{theorem} \label{Theo1} Suppose the map $\psi: \mathbb{R} \to \mathbb{R}$ is Lipschitz with constant $K_1$ and $\eta: \mc \to \mathbb{R}^n$ is Lipschitz with constant $K_2$. If 
\begin{align}
A_0K_1+B_0K_2<1 \nonumber
\end{align}
then for each pair $h \in \mc$, $v \in \mathbb{R}^n$, the boundary value problem
\begin{align}
a_n(t)x^{(n)}(t)+a_{n-1}(t)x^{(n-1)}(t)+\dots+a_0(t) x(t)+\psi(x(t))=h(t) && \nonumber
\end{align}
subject to
\begin{align}
\sum_{j=1}^n\int_0^1 x^{(j-1)}(t)d\omega_{ij}(t)+\eta_i(x)=v_i,  &&1 \leq i \leq n \nonumber 
\end{align}
has a unique solution.
\end{theorem}

\begin{proof}
Suppose $\mathcal{L}(x_0)=0$ for some $0 \neq x_0 \in \mc^n$ and $\{u_1, \dots, u_n\}$ be the basis we chose for $\ker(L)$ above. Since $x_0 \in \ker(L)$ there exists a unique set of constants $c_1, \dots, c_n \in \mathbb{R}$ with $c_i \neq 0$ for some $1 \leq i \leq n$ such that $x_0=\sum_{i=1}^n c_iu_i$. Since $x_0 \in \ker(B)$ we have that 
 \begin{align}
0&=Bx_0 \nonumber \\
&=B\Bigg(\sum_{i=1}^n c_iu_i \Bigg) \nonumber \\
 &=\sum_{i=1}^n c_iBu_i \nonumber 
 \end{align}
contradicting the fact that $\{Bu_1, \dots ,Bu_n\}$ is a linearly independent set in $\mathbb{R}^n$. Therefore, $x_0=0$ and we conclude that $\mathcal{L}: \mc^n \to \mc \times \mathbb{R}^n$ is one-to-one. \\
 \indent Let $h \in \mc$ and $v \in \mathbb{R}^n$. By the general theory of linear scalar ODEs it is well known that the solution $Lx=h$ has at least one solution in $\mc^n$. Let $x_p \in \mc^n$ be the particular solution to this equation given by variation of parameters. That is, 
\begin{align}
x_p(t)=\sum_{k=1}^n u_k(t) \int_0^t \frac{h(s)W_k[u_1, \dots, u_n](s)}{a_n(s)W[u_1,\dots,u_n](s)} ds\nonumber
\end{align}
where $W_k$ denotes the determinant of the matrix obtained by replacing the $k^{th}$ column of the matrix whose determinant is $W$ with $e_n$ (the standard basis vector with a $1$ in the $n^{th}$ slot and $0s$ everywhere else).

 By our assumption that $\{Bu_1, \dots, Bu_n\}$ is a basis for $\mathbb{R}^n$, there exists a unique set of constants $d_1, \dots, d_n$ such that
 \begin{align}
\sum_{i=1}^n d_iBu_i =v-B(x_p). \nonumber
 \end{align}
 Then we have that
 \begin{align}
 L\Bigg( x_p+\sum_{i=1}^n d_iu_i\Bigg)=h+0=h \nonumber
 \end{align}
 and further that 
 \begin{align}
 B\Bigg( x_P+\sum_{i=1}^n d_iu_i\Bigg)&=B(x_p)+ B\Bigg(\sum_{i=1}^n d_iu_i\Bigg) \nonumber \\
 &=B(x_p)+\sum_{i=1}^n d_iBu_i \nonumber \\
 &=B(x_p)+(v-B(x_p)). \nonumber
 \end{align} 
 Therefore, $\mathcal{L}: \mc^n \to \mc \times \mathbb{R}^n$ is a bijection with 
 \begin{align}
 \mathcal{L}^{-1}(h,v)^T =\big( L|_{\ker(B)} \big)^{-1} h+u^T \mathcal{B}^{-1} v \nonumber 
 \end{align}
 where $\big( L|_{\ker(B)} \big)$ is the map $L$ restricted to the kernel of $B$.
 Consequently, we have that $\|\mathcal{L}^{-1}\| \leq \max\{A_0,B_0\}$ where $A_0$ is an upper bound on the norm of the continuous linear map $\big( L|_{\ker(B)} \big)^{-1}$. \\ \\
\indent We now define $\Psi: \mc \to \mc \times \mathbb{R}^n$ by $\Psi(x)=
\begin{bmatrix}
-\psi(x) \nonumber \\
-\eta(x) \nonumber 
\end{bmatrix}$.
Note that since the map from $\mathbb{R}$ to $\mathbb{R}$ given by $t \mapsto \psi(t)$ is Lipschitz with constant $K_1$, then the map from $\mc$ to $\mc$ defined by $x \mapsto \psi \circ x$ is Lipschitz with constant $K_1$. For each pair $(h,v) \in \mc \times \mathbb{R}^n$, we define the map $H_{(h,v)}: \mc \to \mc$ by
\begin{align}
H_{(h,v)}(x)&=\mathcal{L}^{-1} \Psi(x)+\mathcal{L}^{-1}[(h,v)^T]. \nonumber 
\end{align}
Let $x_1, x_2 \in \mc$. Then we have that 
\begin{align}
\|H_{(h,v)}(x_1)-H_{(h,v)}(x_2)\|&=\|\mathcal{L}^{-1} \Psi(x_1)-\mathcal{L}^{-1} \Psi(x_2)\| \nonumber \\
&\leq A_0 \| \psi(x_1)-\psi(x_2)\|+B_0 \|\eta(x_1)-\eta(x_2)\| \nonumber \\
&\leq (A_0K_1+B_0K_2)\|x_1-x_2\|. \nonumber
\end{align}
Then $H_{(h,v)}$ is a contraction on $\mc$ and so by Banach's fixed point theorem it has a unique fixed point $x_0 \in \mc$. Since $\mathcal{L}^{-1}$ maps into $\mc^n$, we have that $x_0 \in \mc^n$. Therefore, there exists a unique $x_0 \in \mc^n$ satisfying $\mathcal{L}(x)-\Psi(x)=(h,v)^T$. Since $h \in \mc$ and $v \in \mathbb{R}^n$ were arbitrary, we conclude that the operator $(\mathcal{L}-\Psi):\mc^n \to \mc \times \mathbb{R}^n$ is invertible.  From this it follows that for each pair $(h,v) \in \mc \times \mathbb{R}^n$, the boundary value problem
\newpage

\begin{align}
a_n(t)x^{(n)}(t)+a_{n-1}(t)x^{(n-1)}(t)+\dots+a_0(t) x(t)+\psi(x(t))=h(t) && \nonumber
\end{align}
subject to
\begin{align}
\sum_{j=1}^n\int_0^1 x^{(j-1)}(t)d\omega_{ij}(t)+\eta_i(x)=v_i, && 1 \leq i \leq n \nonumber 
\end{align}
has exactly one solution.
 \qed \medskip
\end{proof}

The following lemma establishes an important result regarding the map $\mathcal{L}^{-1}: \mc \times \mathbb{R}^n \to \mc$. The importance of this lemma will become apparent when we provide our conditions for the solvability of $\eqref{de}-\eqref{bc}$.

\begin{lemma}
The operator $\mathcal{L}^{-1}:\mc \times \mathbb{R}^n \to \mc$ is compact.
\end{lemma}
\begin{proof}
Let $M>0$ and define $S=\{(h,v) \in \mc \times \mathbb{R}^n: \|h\|+|v| \leq M\}$. Let $(h,v) \in S$. Then
\begin{align}
\|\mathcal{L}^{-1}(h,v)\| &\leq \big\| \mathcal{L}^{-1} \big\| ( \|h\|+|v| )\nonumber \\
&\leq \max\{A_0,B_0\} (\|h\|+|v|) \nonumber \\
&\leq \max\{A_0,B_0\} M. \nonumber
\end{align}
Therefore the set $\{\mathcal{L}^{-1}(S)\}$ is bounded. We now wish to show this set forms an equicontinuous set of functions. \\ \\
Let $\varepsilon>0$, and let $\delta=\frac{\varepsilon}{\max\{A_0,B_0\}M}$. Then for any $(h,v) \in \mc \times \mathbb{R}^n$ and $t_1, t_2 \in [0,1]$ with $|t_1-t_2|<\delta$,

\begin{multline*} 
\big| \big[ \mathcal{L}^{-1} (h,v) \big](t_1)- \big[ \mathcal{L}^{-1} (h,v) \big](t_2) \big| \\
= \Bigg| \Big[ \big(L |_{\ker(B)} \big)^{-1}h \Big](t_1)-\Big[ \big(L |_{\ker(B)} \big)^{-1}h \Big](t_2) +u^T(t_1) \mathcal{B}^{-1}v(t_1)-u^T(t_2) \mathcal{B}^{-1}v(t_2) \Bigg| \nonumber 
 \end{multline*} 
 \begin{align}
&\leq \Bigg| \Big[ \big(L |_{\ker(B)} \big)^{-1}h \Big] (t_1)-\Big[ \big(L |_{\ker(B)} \big)^{-1}h \Big](t_2) \Bigg|+\Big| u^T (t_1)\mathcal{B}^{-1}v(t_1)-u^T(t_2) \mathcal{B}^{-1}v(t_2) \Big| \nonumber \\
&\leq \max\{A_0,B_0\}M|t_1-t_2| \nonumber \\
&< \varepsilon. \nonumber
\end{align}

Therefore the set of functions $\{\mathcal{L}^{-1}(S)\}$ is equicontinous and we conclude that $\mathcal{L}^{-1}: \mc \times \mathbb{R}^n \to \mc$ is compact by the Arzel\'{a}--Ascoli theorem. \qed \medskip
\end{proof}

 \indent Recall that in the proof of theorem \ref{Theo1}, we established that the operator $\mathcal{L}-\Psi$ is a bijection from $\mc^n$ onto $\mc \times \mathbb{R}^n$. We now state important properties of $(\mathcal{L}-\Psi)^{-1}$ under the conditions of theorem \ref{Theo1}.  The proof of the first corollary follows directly from corollary 2.3.2 in \cite{dra}. \\

 \begin{corollary}
 Suppose the conditions of theorem \ref{Theo1} hold. Then the map $(\mathcal{L}-\Psi)^{-1}: \mc \times \mathbb{R}^n \to \mc^n$ is Lipschitz continuous with constant 
 \begin{align}
 K \equiv \frac{\max\{A_0,B_0\}}{1-(A_0K_1+B_0K_2)}. \nonumber \\
 \nonumber
 \end{align}
 \end{corollary}
 
 \begin{corollary}
 Under the conditions of theorem \ref{Theo1}, the map $(\mathcal{L}-\Psi)^{-1}: \mc \times \mathbb{R}^n \to \mc^n$ is compact. This follows from the fact that we can write 
 \begin{align}
 (\mathcal{L}-\Psi )^{-1}=\mathcal{L}^{-1} \Big(\Psi \circ (\mathcal{L}-\Psi )^{-1}+I \Big). \nonumber
 \end{align}
 Therefore it is clear from this representation that $(\mathcal{L}-\Psi)^{-1}$ is compact as the composition of a compact operator with a continuous one.
 \end{corollary}

\indent We now proceed to establish conditions for the solvability of the boundary value problem
\begin{align}
a_n(t)x^{(n)}(t)+a_{n-1}(t)x^{(n-1)}(t)+\dots+a_0(t) x(t)+\psi(x(t))=G(x)(t) \nonumber
\end{align}
subject to the boundary conditions
\begin{align}
\sum_{j=1}^n\int_0^1 x^{(j-1)}(t)d\omega_{ij}(t)+\eta_i(x) &=\phi_i(x)  \nonumber
\end{align}
 for $1 \leq i \leq n$. \\
 \indent We define $\phi:\mc \to \mathbb{R}^n$ by $\phi=
\begin{bmatrix}

\phi_1 \nonumber \\
\phi_2 \nonumber \\
\dots \nonumber \\
\phi_n \nonumber 
\end{bmatrix}$
 and $\mathcal{G}: \mc \to \mc \times \mathbb{R}^n$ by
$\mathcal{G}=\Bigg[
\begin{array}{c}
G \\
\phi \\
\end{array}
\Bigg]$.
\\

 In doing so, we are now ready to state sufficient conditions under which we can guarantee the existence of at least one solution to the nonlinear boundary value problem $(3)-(4)$. \\
 
 \indent Before stating theorem 2, define $M_0$ as the norm of the unique solution  to the boundary value problem
 \begin{align}
a_n(t)x^{(n)}(t)+a_{n-1}(t)x^{(n-1)}(t)+\dots+a_0(t) x(t)+\psi(x(t))=0 \nonumber
\end{align}
subject to the boundary conditions
\begin{align}
\sum_{j=1}^n\int_0^1 x^{(j-1)}(t)d\omega_{ij}(t)+\eta_i(x) &=0 \nonumber
\end{align}
for $1 \leq i \leq n$.

\begin{theorem} \label{theo2}
Suppose the map $\psi: \mathbb{R} \to \mathbb{R}$ is Lipschitz with constant $K_1$ and $\eta: \mc \to \mathbb{R}^n$ is Lipschitz with constant $K_2$. If 
\begin{align}
A_0K_1+B_0K_2<1 \nonumber
\end{align}
and there exists a constant $M$ such that for $\|x\| \leq M$, $\| \mathcal{G}(x)\| \leq K^{-1}(M-M_0)$. Then there exists a solution to the boundary value problem
\begin{align}
a_n(t)x^{(n)}(t)+a_{n-1}(t)x^{(n-1)}(t)+\dots+a_0(t) x(t)+\psi(x(t))=G(x)(t) \nonumber 
\end{align}
subject to the boundary conditions
\begin{align}
\sum_{j=1}^n\int_0^1 x^{(j-1)}(t)d\omega_{ij}(t)+\eta_i(x) &=\phi_i(x)  \nonumber
\end{align}
for $1 \leq i \leq n$.
\end{theorem}

\begin{proof}
Note that the map $(\mathcal{L}-\Psi)^{-1} \mathcal{G}:\mc \to \mc$ is compact as the composition of a compact operator with a continuous one. Define $B=\{ z \in \mc: \|z\| \leq M\}$. Let $x \in B$. Then
\begin{align}
\|(\mathcal{L}-\Psi)^{-1} \mathcal{G}(x) \| &\leq \|(\mathcal{L}-\Psi)^{-1} \mathcal{G}(x) -M_0 \|+M_0 \nonumber \\
&\leq K\|\mathcal{G}(x)\| +M_0\nonumber \\
&\leq K(K^{-1}(M-M_0))+M_0 \nonumber \\
&= M. \nonumber
\end{align}
Since $(\mathcal{L}-\Psi)^{-1} \mathcal{G}(B) \subseteq B$ and $B$ is clearly closed, bounded, and convex we have that $(\mathcal{L}-\Psi)^{-1} \mathcal{G}$ has at least one fixed point in $\mc$ by Schauder's fixed point theorem. That is, there exists at least one $x_0 \in \mc$ such that $(\mathcal{L}-\Psi)^{-1} \mathcal{G}(x_0)=x_0$. Since $(\mathcal{L}-\Psi)^{-1}$ maps into $\mc^n$, we have that $x_0$ must be an element of $\mc^n$. This is equivalent to there existing at least one $x_0 \in \mc^n$ such that $\mathcal{L}(x_0)-\Psi(x_0)=\mathcal{G}(x_0)$.
\qed \medskip
\end{proof}

The following corollary is immediate:
\begin{corollary}
Suppose the map $\psi: \mathbb{R} \to \mathbb{R}$ is Lipschitz with constant $K_1$ and $\eta: \mc \to \mathbb{R}^n$ is Lipschitz with constant $K_2$.  If 
\begin{align}
A_0K_1+B_0K_2<1 \nonumber
\end{align} and
\begin{align}
\lim_{\|x\| \to \infty} \frac{\|\mathcal{G}(x)\|}{\|x\|}=0 \nonumber
\end{align} 
then the boundary value problem $\eqref{de}-\eqref{bc}$ has a solution.
\end{corollary}

In the next section, we consider advantages this framework provides us in cases where we attempt to analyze problems that are seemingly well-suited to using the framework outlined in \cite{ja2} and \cite{jsuar1}.

\section{Comparison to Previous Results}
To view advantages of using this framework as opposed to previous results, consider another set of special cases of the general boundary value problem $\eqref{de}-\eqref{bc}$. That is, problems already in self-adjoint form. This is a necessity if we are to attempt to use the analysis of \cite{ja2} and \cite{jsuar1}. \\

\begin{remark}
\indent Consider differential equations on $[0,1]$ of the form:
\begin{align}
(p(t)x'(t))'+q(t)x(t)+\psi(x(t))=G(x(t)) \label{sl}
\end{align}
subject to the boundary conditions
\begin{align}
\alpha x(0)+\beta x'(0)+ \sum_{j=1}^2 \int_0^1  x^{(j-1)}(t) d \omega_{1j} (t)+\eta_1(x)=0  \nonumber \\
\label{slbc} \\
\gamma x(1)+\delta x'(1)+\sum_{j=1}^2 \int_0^1 x^{(j-1)}(t) d \omega_{2j}(t)+\eta_2(x)=0 \nonumber
\end{align}
where $\psi: \mathbb{R} \to \mathbb{R}$ is Lipschitz, $\alpha^2+\beta^2 \neq 0$, $\gamma^2+\delta^2 \neq 0$, $\eta_1$ and $\eta_2$ are nonlinear Lipschitz functions from $\mc^2$ into $\mathbb{R} $. The function $\omega_{ij}:[0,1] \to \mathbb{R}$ is a function of bounded variation for $i=1,2$ and $j=1,2$. We assume that $p, p'$, and $q$ are continuous, $p(t)>0$ for all $t \in [0,1]$. We assume the map $G$ is a continuous function from $\mc$ into $\mc$ satisfying
\begin{align}
\lim_{\|x\| \to \infty} \frac{\|G(x)\|}{\|x\|}=0 .\nonumber
\end{align}   \\
\indent Using results from \cite{ja2} and \cite{jsuar1}, we would have to treat the linear integral boundary conditions appearing in $\eqref{slbc}$ as part of the nonlinear component of the problem. Let $\hat{K}_2$ be the Lipschitz constant of the map defined by
\begin{align}
x \mapsto \Bigg[
\begin{array}{c}
\sum_{j=1}^2 \int_0^1  x^{(j-1)}(t) d \omega_{1j} (t)+\eta_1(x) \nonumber \\
\sum_{j=1}^2 \int_0^1  x^{(j-1)}(t) d \omega_{2j} (t) +\eta_2(x) \nonumber
\end{array}
\Bigg]
\end{align}
with respect to the norms used in those previous papers. Suppose $\{u_1,u_2\}$ is a basis for the solution space of 
\begin{align}
(p(t)x'(t))'+q(t)x(t)=0 \nonumber
\end{align}
and without loss of generality suppose that 
\begin{align}
\Bigg(\int_0^1 |u_1(t)|^2 dt \Bigg)^{\frac{1}{2}}+\Bigg(\int_0^1 |u_2(t)|^2 dt \Bigg)^{\frac{1}{2}} \leq 1. \nonumber
\end{align}
Further, suppose the $2 \times 2$ matrix 
$\hat{\mathcal{B}}=\Bigg[
\begin{array}{cc}
\alpha u_1(0)+\beta u_1'(0),&\alpha u_2(0)+\beta u_2'(0) \nonumber \\
\gamma u_1(1)+\delta u_1'(1),&\gamma u_2(1)+\delta u_2'(1) \nonumber
\end{array} 
\Bigg]$ is invertible.
Then if
\begin{align}
\Bigg(\sup_{v \in \mathbb{R}^2} |\hat{\mathcal{B}}^{-1} v| \Bigg) \hat{K}_2  \geq 1 \nonumber
\end{align}
it would be impossible to establish the existence of solutions to $\eqref{sl}-\eqref{slbc}$ using any of the results appearing in \cite{ja2} or \cite{jsuar1}. 

When we formulate $\eqref{sl}-\eqref{slbc}$ within the framework presented in section 3, it is clear that if the map
\begin{align}
x \mapsto 
\begin{bmatrix}
(p(t)x'(t))'+q(t)x(t) \nonumber \\
\alpha x(0)+\beta x'(0)+\sum_{j=1}^2 \int_0^1  x^{(j-1)}(t) d \omega_{1j} (t) \nonumber \\
\gamma x(1)+\delta x'(1)+\sum_{j=1}^2 \int_0^1  x^{(j-1)}(t) d \omega_{2j} (t)  \nonumber
\end{bmatrix}
\end{align}
is a bijection from its domain onto $\mc \times \mathbb{R}^2$ then the boundary value problem would have a solution provided the Lipschitz constants for $\eta_1, \eta_2$, and $\psi$ are sufficiently small.
The magnitude of the linear integral boundary conditions is completely irrelevant.
\end{remark}

We would now like to point out that the map $G$ that we have just described can be generated in a variety of ways. For example $G$ could be of the form
\begin{align}
G(x)(t)=g(x(t)) \nonumber
\end{align}
where $g: \mathbb{R} \to \mathbb{R}$ is continuous or of the type
\begin{align}
G(x)(t)= \int_0^1 k(t,x(s)) ds \nonumber 
\end{align}
where $k: \mathbb{R}^2 \to \mathbb{R}$. \\

In addition to the advantages that we have just discussed, we would like to point out that if the nonlinearities $\eta_i$ appearing above are of the form
\begin{align}
\eta_i=\sum_{j=1}^N f_{i,j} (x(t_j)) \nonumber
\end{align}
then in order to use results appearing in \cite{ja2} or \cite{jsuar1} it must be assumed that the operator $G$ is compact. This restriction is no longer present when using the results that we have just presented.


\end{document}